\documentclass[12pt]{article}
\usepackage{amssymb, amscd, amsthm, amsmath,latexsym, amstext,mathrsfs,verbatim}
\usepackage[all]{xy}
\usepackage{graphicx,enumerate}
\usepackage{url}

\def\op{{\rm op}}
\def\ddA{{\rm A}}

\def\Br{{\rm Br}}

\def\BrM{{\rm BrM}}

\def\ddI{{\rm I}}

\def\ddC{{\rm C}}

\def\ddG{{\rm G}}

\def\hE{{\hat E}}
\def\eps{{\epsilon}}
\def\alp{{\alpha}}

\newcommand{\cA}{\mathcal{A}}

\newcommand{\N}{\mathbb N}

\newcommand{\R}{\mathbb R}
\newcommand{\Z}{\mathbb Z}

\def\mod{\mathrm{mod}}

\numberwithin{equation}{section}

\newtheorem{lemma}{Lemma}[section]

\newtheorem{prop}[lemma]{Proposition}
\newtheorem{thm}[lemma]{Theorem}

\theoremstyle{definition}

\newtheorem{defn}[lemma]{Definition}

\theoremstyle{remark}
\newtheorem{rem}[lemma]{Remark}

\def\b{\beta}

\def\alp{\alpha}

\begin{document}
\title{Brauer algebras of type $\ddI_2^n$($n\ge 5$)}
\author{Shoumin Liu}
\date{}
\maketitle

\begin{abstract}
We will present an algebra  related to the Coxeter group
of type $\ddI_2^n$ which can be  taken as a twisted subalgebra in
Brauer algebra of type $\ddA_{n-1}$. Also we will describe some properties
of this algebra.
\end{abstract}

\begin{section}{Introduction}
From studying the invariant theory for orthogonal groups,
Brauer discovered  Brauer algebras of type $\ddA$ in \cite{Brauer1937};  Cohen, Frenk and Wales extended it to the
definition of simply laced type in \cite{CFW2008}.
M\"{u}hlherr  described how to get Coxeter group of type $\ddI_2^{n}$ by twisting Coxeter group of  type $\ddA_{2n-1}$ in \cite{M1993}.
Here we will apply the similar approach as M\"{u}hlherr on $\Br(\ddA_{n-1})$, to get $\Br(\ddI_2^n)$. We give  the definition
of $\Br(\ddI_2^n)$ as follows.
\begin{defn}\label{0.1}
\emph{The Brauer algebra of type} $\ddI_2^{2m}$ for $m\in \N_{>2}$, denoted by $\Br(\ddI_2^{2m})$,
is a unital associative $\Z[\delta^{\pm 1}]$-algebra generated by $r_0$, $r_1$, $e_{0}$ and  $e_1$ subject to
the following relations and a   set $\Theta\subset \N$ consisting of
$\kappa_i$, $\eta_{j}$, $\xi_{j}$  $\theta_{j}$, where $i=0$, $1$, $j=1$, $\ldots$, $[m/2]$.
Symbols $[r_0r_1\cdots]_t$ and $[r_1r_0\cdots]_t$ stand for  words of length $t$ with $r_0$ and $r_1$ iterated.

\begin{eqnarray}
r_{i}^{2}&=&1, \    \label{0.1.3}
\\
r_ie_i &= & e_ir_i=e_i, \,\, \label{0.1.4}
\\
e_{i}^{2}&=&\delta^{\kappa_i} e_{i},\,   \label{0.1.5}
\\
e_1e_0e_1&=&\delta e_1,            \label{0.1.6}
\\
e_0[r_1r_0\cdots]_{2m-1}&=&[r_1r_0\cdots]_{2m-1}e_0,   \label{0.1.7}
\\
e_1[r_0r_1\cdots]_{2m-1}&=&e_1,     \label{0.1.8}
\\
\left[ r_0r_1\cdots\right]_{2m-1} e_1&=&e_1,      \label{0.1.9}
\\
e_0[r_1r_0\cdots]_{2k}e_1&=&\delta^{\theta_{k}}e_0e_1,\, \quad  0\le k\le [m/2]        \label{0.1.10}
\\
e_1[r_0r_1\cdots]_{2k}e_0&=&\delta^{\theta_{k}}e_1e_0,\, \quad  0\le k\le [m/2]                               \label{0.1.11}
\\
e_1[r_0r_1\cdots]_{2k-1}e_1&=&\delta^{\eta_{k}}e_1,\,\quad  0\le k\le [m/2] \label{0.1.12}
\\
\left[r_1r_0\cdots\right]_{2m}&=&[r_0r_1\cdots]_{2m}                                              \label{0.1.20}
\end{eqnarray}
and when $2k\le m$, let $l={\rm lcm}(k,m)$,
\begin{eqnarray}
&&e_0[r_1r_0\cdots]_{2k-1}e_0=\delta^{\xi_{k}}[r_1r_0\cdots]_{2m-1}e_0, \quad  \frac{l}{m}, \frac{l}{k}\quad \mbox{odd},\label{0.1.13}
\\
&&e_0[r_1r_0\cdots]_{2k-1}e_0=\delta^{\xi_{k}}e_0,\quad\quad\quad \frac{l}{m}\,\quad \mbox{even}, \label{0.1.14}
\\
&&e_0[r_1r_0\cdots]_{2k-1}e_0=\delta^{\xi_{k}}e_0e_1e_0,\quad \frac{l}{m}\quad \mbox{odd},\, \frac{l}{k}\quad \mbox{even}.\label{0.1.15}
\end{eqnarray}
 The submonoid
of the multiplicative monoid of $\Br(\ddI_2^{2m})$
generated by $\delta$, $\{r_i\}_{i=0}^{1}$ and $\{e_i\}_{i=0}^{1}$ is
denoted by $\BrM(\ddI_2^{2m})$. This is the monoid of monomials in
$\Br(\ddI_2^{2m})$.
\end{defn}

\begin{defn}\label{0.2}
\emph{The Brauer algebra of type} $\ddI_2^{2m-1}$ for $m\in \N_{>2}$, denoted by $\Br(\ddI_2^{2m-1})$,
is a unital associative $\Z[\delta^{\pm 1}]$-algebra generated by $r_0$, $r_1$, $e_{0}$ and  $e_1$ subject to
the following relations and  a  set $\Theta\subset \N$ consisting of
$\kappa_i$,  $\xi_{j}$  with $j=1$, $\ldots$, $m$, $i=0$, $1$, $\kappa_0=\kappa_1$.
\begin{eqnarray}
r_{i}^{2}&=&1,   \label{0.2.3}
\\
r_ie_i &= & e_ir_i=e_i, \, \label{0.2.4}
\\
e_{i}^{2}&=&\delta^{\kappa_i} e_{i},\,\,  \label{0.2.5}
\\
\left[r_0r_1\cdots\right]_{2m-2}e_0&=&e_1[r_0r_1\cdots]_{2m-2},   \label{0.2.6}
\\
e_0[r_1r_0\cdots]_{2k-1}e_0&=&\delta^{\xi_{k}}e_0,\quad  0<k<m \label{0.2.7}
\\
\left[r_1r_0\cdots\right]_{2m-1}&=&[r_0r_1\cdots]_{2m-1}.                                              \label{0.2.8}
\end{eqnarray}

 The submonoid
of the multiplicative monoid of $\Br(\ddI_2^{2m-1})$
generated by $\delta$, $r_0$, $r_1$, $e_0$, $e_1$ is
denoted by $\BrM(\ddI_2^{2m-1})$. This is the monoid of monomials in
$\Br(\ddI_2^{2m-1})$.
\end{defn}

It is well known that Coxeter group of type $\ddI_2^{n}$, denoted by $W(\ddI_2^n)$ can be gotten as a subgroup from
Coxeter group of type $\ddA_{n-1}$ denoted by $W(\ddA_{n-1})$ for an special partition on the Coxeter diagram of type $\ddA_{n-1}$.
The following is the partition to get $\ddI_2^6$ from $\ddA_{5}$.
\begin{figure}[!htb]
\begin{center}
\includegraphics[width=.9\textwidth,height=.2\textheight]{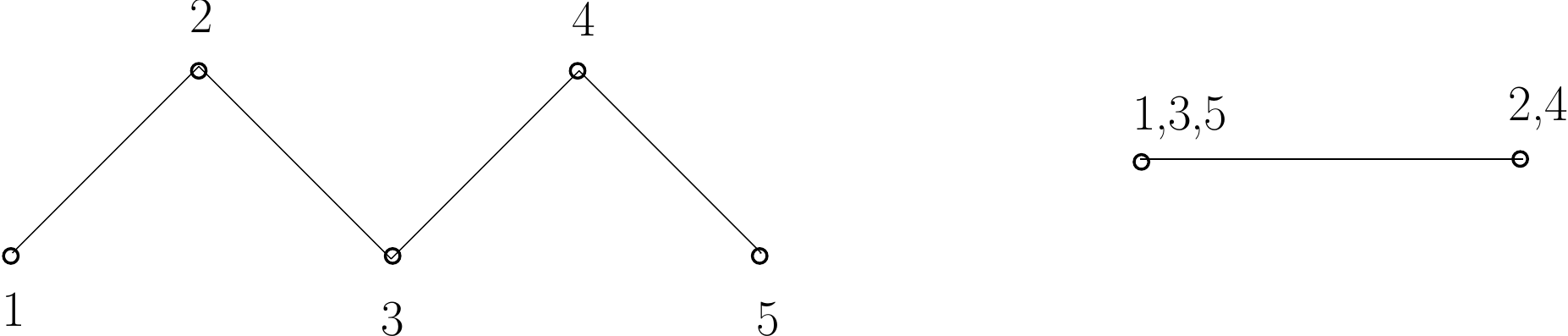}
\end{center}
\end{figure}
The main theorem of this paper can be stated as the following.
\begin{thm}\label{main}
For $n>4$, there is an  algebra isomorphism
$$\phi:\, \Br(\ddI_2^{n})\longrightarrow \Br(\ddA_{n-1})$$
determined by $\phi(r_0)=\prod_{i\, even}^{0<i<n}R_i$, $\phi(r_1)=\prod_{i\, odd}^{0<i<n}R_i$,
$\phi(e_0)=\prod_{i\, even}^{0<i<n}E_i$ and $\phi(e_1)=\prod_{i\, odd}^{0<i<n}R_i$ when each parameter
 in $\Theta$ takes special value in $\N$ to make $\phi$ an algebra homomorphism.
Furthermore,\\
 $rank_{\Z[\delta^{\pm 1}]} \Br(\ddI_2^{n})=$
$\begin{cases}
2n+n^2,\quad \text{if n is odd,}\\
2n+\frac{3}{2}n^2,\quad \text{if n is even.}
\end{cases}$
 \end{thm}
 This paper is included as Chapter $5$ in the author's PhD thesis \cite{L2012}.
\end{section}
\section{An interesting elementary problem}
Suppose that $k$, $m\in \mathbb{N}$ are such that $1<2k\leq m$.
There is a box in the $x$, $y$ plane $\mathbb{R}^2$ fixed by four lines $x=1$, $x=2m$, $y=2k-\frac{1}{2}$, and $y=-\frac{1}{2}$.
Imagine you have a particle, which starts to move from $(1,2k-1)$ with slope $-1$;  when it touches the bottom (the top),
it will be reflected with the
bottom (the top) as a mirror; but when it touches the right (left) wall, it  first goes down (up)   1 unit vertically, if it comes at the wall from the top (bottom), and continues its path with the wall as the mirror;
it stops if it reaches  the points $(1, 0)$, $(2m,0)$, $(1,2k-1)$, or $(2m,2k-1)$.
For different values of $k$, $m$, the problem is to decide at which point the particle stops.
One example is Figure \ref{m5k2}, when $m=5$, $k=2.$

\begin{figure}[!htb]
\begin{center}
\includegraphics[width=.4\textwidth,height=.2\textheight]{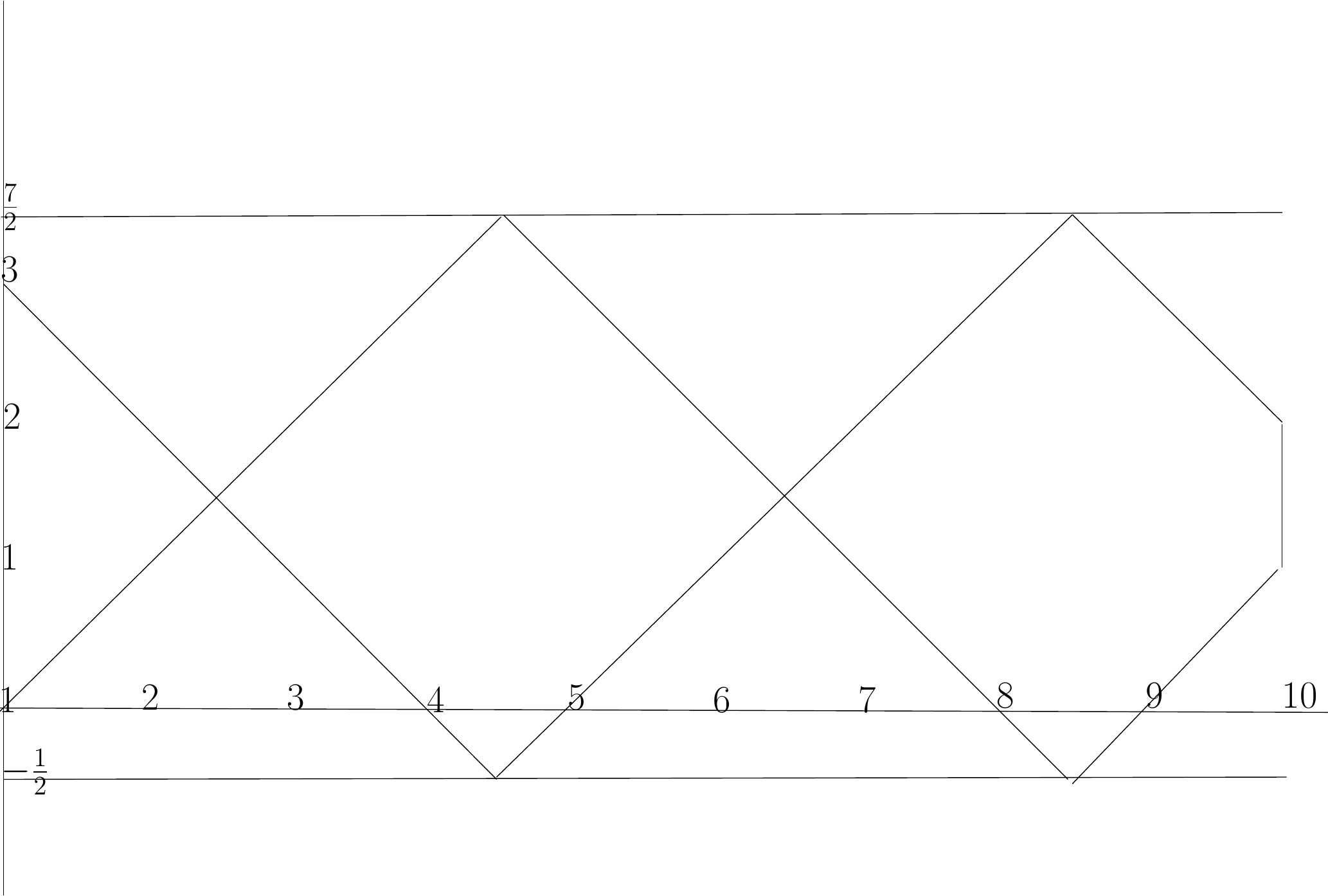}
\caption{case for $m=5$, $k=2$}\label{m5k2}
\end{center}
\end{figure}

To solve the problem, we  unfold its  path  by "penetrating" the walls,
which means that when the particle touches the right wall for the first time, we change the
vertical step into one move of slope $-1$ when coming from  the top (slope $1$ when coming from the bottom) with Euclid length $\sqrt{2}$, or, in other words, we do not change its moving at the wall, and we see that the path of the particle in the region between $x=2m+1$ and $x=4m$ is
just the reverse of the path of the particle when it goes from the right to the left for the first time. The algorithm can be similarly extended
at the left wall to make the unfolded path look like the graph of a function of a single variable.
 It can be verified that when  it passes the point with the $x$-coordinate being a multiple
of $2m$, the movement stops. It can be seen that before it stops, the path in $[2tm+1, 2(t+1)m]$ is just a copy
of a particle path in the above box of :
\begin{enumerate}[(i)]
\item the $\frac{t+2}{2}$th path from the left wall to the right wall if $t$ is even,
\item or $\frac{t+1}{2}$th path from the right wall to the the left wall if $t$ is odd.
\end{enumerate}
 Therefore
the above trick is just that  we  draw the picture on  folded paper, then  we  unfold this and see a simple picture in which the original
problem can be tackled.
Here is an example for $m=5$, $k=2$ in  Figure \ref{m5k2^2} being the unfolded case of Figure \ref{m5k2}.

\begin{figure}[!htb]
\begin{center}
\includegraphics[width=.8\textwidth,height=.2\textheight]{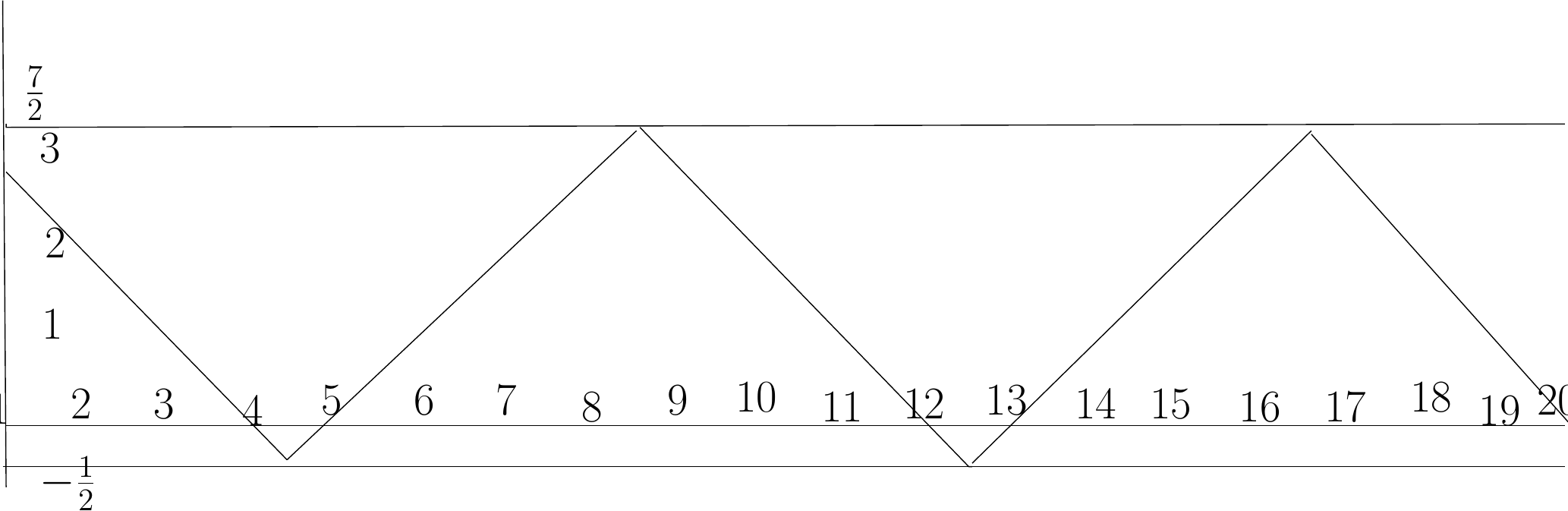}
\caption{the unfolded path  for  $m=5$, $k=2$}\label{m5k2^2}
\end{center}
\end{figure}
 \begin{lemma}\label{km} Let $l={\rm lcm}(k,m)$.
 The particle stops in the unfolded path when it moves $2l-1$ for its $x$-coordinate. Furthermore
 \begin{enumerate}[(i)]
 \item when $\frac{l}{m}$ is even, the particle stops at $(1,0)$;
  \item when $\frac{l}{m}$ and  $\frac{l}{k}$ are odd,  the particle stops at $(2m,0)$;
  \item when $\frac{l}{m}$ is odd and  $\frac{l}{k}$ is even, the particle stops at $(2m,2k-1)$.
  \end{enumerate}
 \end{lemma}
 \begin{proof} By prolonging the path  at the beginning and the ending, respectively,  by half unit for $x$-coordinate to complete a period,
  we can consider the particle starting
 from the top and stopping at the top or the bottom.
 By observing the unfolded path of the particle,  each time it goes from the top ceiling to the bottom ground or from the bottom to
 the top, the $x$-coordinate is increased by $2k$, so the first conclusion follows naturally. The
other two conclusions hold easily by basic number theory knowledge about congruence. Furthermore, elementary number theory tells us that
 $\frac{l}{m}$ and  $\frac{l}{k}$ can not both be even, which implies that the particle never stops at $(1, 2k-1)$.
 \end{proof}

\section{The map $\phi$ inducing a homomorphism}

In order to avoid confusion
with the above generators, the symbols of \cite{CFW2008} have been capitalized.

\begin{defn}\label{1.1}
Let $Q$ be a graph. The Brauer monoid $\BrM(Q)$ is the monoid
generated by the symbols $R_i$ and $E_i$, for  each node $i$ of $Q$ and $\delta$,
$\delta^{-1}$ subject to the following relation, where
$\sim$ denotes adjacency between nodes of $Q$.

\begin{equation}\delta\delta^{-1}=1     \label{1.1.1}
\end{equation}
\begin{equation}R_{i}^{2}=1          \label{1.1.2}
\end{equation}
\begin{equation}R_iE_i=E_iR_i=E_i     \label{1.1.3}
\end{equation}
\begin{equation}E_{i}^{2}=\delta E_{i}   \label{1.1.4}
\end{equation}
\begin{equation}R_iR_j=R_jR_i, \,\, \mbox{for}\, \it{i\nsim j} \label{1.1.5}
\end{equation}
\begin{equation}E_iR_j=R_jE_i,\,\, \mbox{for}\, \it{i\nsim j}  \label{1.1.6}
\end{equation}
\begin{equation}E_iE_j=E_jE_i,\,\, \mbox{for}\, \it{i\nsim j}    \label{1.1.7}
\end{equation}
\begin{equation}R_iR_jR_i=R_jR_iR_j, \,\, \mbox{for}\, \it{i\sim j}  \label{1.1.8}
\end{equation}
\begin{equation}R_jR_iE_j=E_iE_j ,\,\, \mbox{for}\, \it{i\sim j}       \label{1.1.9}
\end{equation}
\begin{equation}R_iE_jR_i=R_jE_iR_j ,\,\, \mbox{for}\, \it{i\sim j}     \label{1.1.10}
\end{equation}
The Brauer algebra $\Br(Q)$ is  the the free $\Z$-algebra for Brauer monoid  $\BrM(Q)$.
 \end{defn}

In \cite{Brauer1937}, Brauer gives a diagram description for a basis of
Brauer monoid of type $\ddA_t$, which is a monoid consisting of diagrams with $2t+2$
dots and $t+1$ strands, where each dot is connected by a unique strand to another dot.

Here we suppose the $2t+2$ dots have coordinates $(i, 0)$ and $(i,1)$ in $\R^{2}$ with
 $1\leq i\leq t+1$. The multiplication of two diagrams is given by concatenation, where any closed loops formed are replaced by a factor of $\delta$. Henceforth, we identify $\Br(\ddA_{t})$ with its
 diagrammatic version. It is a free algebra over $\Z[\delta^{\pm1}]$ of rank $(t+1)!!$, the product of the first $t+1$ positive odd integers.

 we revise the root system of the Coxeter group of type $\ddA_t$, focussing on special collections of mutually orthogonal positive roots called admissible sets.  Also, the notion of height for elements of the Brauer algebra $Br(\ddA_{t})$ is introduced and discussed.

\begin{defn}\label{defn:Phi}
Let $t\ge1$.  The root system of the Coxeter group $W(\ddA_{t})$ of type
$\ddA_{t}$ is denoted by $\Phi$.  It is realized as $\Phi :=
\{\eps_i-\eps_j\mid 1\le i,j\le t+1,\ i\ne j\}$ in the Euclidean space
$\R^{t+1}$, where $\eps_i$ is the $i^\mathrm{th}$ standard basis vector.  Put $\alp_i
:= \eps_i-\eps_{i+1}$.  Then $\{\alpha_{i}\}_{i=1}^{t}$ is called the set of
simple roots of $\Phi$.  Denote by $\Phi^+$ the set of positive roots in
$\Phi$ with respect to these simple roots; that is, $\Phi^+ := \{\eps_i-\eps_j\mid 1\le i<j\le t+1\}$.
\end{defn}

We have seen that, up to powers of $\delta$, the monomials of $\Br(\ddA_t)$
correspond to Brauer diagrams. In order to work with the tops and bottoms of
Brauer diagrams, we introduce the following notion.

\begin{defn}
\label{df:cA}
Let $\cA$ denote the collection of all subsets of $\Phi^+$ containing
mutually orthogonal positive roots.
Members of $\cA$ are called \emph{admissible sets}.
\end{defn}

An admissible set $B$ corresponds to a Brauer diagram top in the
following way: for each $\b\in B$, where $\b = \eps_i-\eps_j$ for some $i,j\in\{1,\ldots,t+1\}$, draw a horizontal strand in the corresponding Brauer diagram top from the dot $(i,1)$ to the dot $(j,1)$. All
horizontal strands on the top are obtained this way, so there are precisely
$|B|$ horizontal strands.

For any $\beta\in\Phi^+$ and $i\in\{1,\ldots,t\}$, there exists a $w\in W(\ddA_t)$ such
that $\beta = w\alpha_i$. Then $E_\beta := wE_iw^{-1}$ is well defined
(see \cite{CFW2008}).
If $\beta,\gamma\in\Phi^+$ are mutually orthogonal, then $E_\beta$ and
$E_\gamma$ commute (see Lemma 4.3 of \cite{CFW2008}). Hence, for $B\in\cA$, we can define
$$\hE_B := \delta^{-|B|}\prod_{\beta\in B} E_\beta.$$
This is an idempotent element of the Brauer monoid.

In \cite{CFW2008}, an action of the Brauer monoid
$\BrM(\ddA_{t})$ on the collection $\cA$ is defined as follows.
The generators $\{R_{i}\}_{i=1}^{t}$ act
by the natural action of Coxeter group elements on its
root sets, where negative roots are negated so as to obtain positive roots,
and the action of $\{E_{i}\}_{i=1}^{t}$ is defined below.
\begin{equation}
E_i B :=\begin{cases}
B & \text{if}\ \alpha_i\in B, \\
B\cup \{\alpha_{i}\} & \text{if}\ \alpha_i\perp B,\\
R_\beta R_i B & \text{if}\ \beta\in B-\alpha_{i}^{\perp}
\end{cases}
\end{equation}

Alternatively, this action can be described as follows: complete
the top corresponding to $B$ into a Brauer diagram $b$, without increasing the number of horizontal srands at the top.
Now $aB$ is the top of
the Brauer diagram $ab$.  We will make use of this action in order to
provide a normal form for elements of $\BrM(\ddA_t)$.
\begin{rem}\label{rem:opp}
There is an anti-involution on $\Br(\ddA_t)$ determined by
$$\gamma_1\cdots\gamma_t\ \mapsto\ \gamma_t\cdots\gamma_1$$
on products of generators of $\Br(\ddA_n)$. We denote it by $x\mapsto x^{\op}.$
\end{rem}
\begin{lemma} The map $\phi$ defined on the generators in Theorem \ref{main} induces a $\Z[\delta^{\pm 1}]$-algebra homomorphism.
\end{lemma}
\begin{proof} We deal first with $n$ is odd.
In Definition \ref{0.2}, the nontrivial relations to be verified are  just (\ref{0.2.6}), (\ref{0.2.7}), and (\ref{0.2.8}).
By \cite{M1993}, the relation (\ref{0.2.8}) holds for the inspection of the images of  generators.
The relation (\ref{0.2.7}) follows by observing the diagrams in $\Br(\ddA_{2m-2})$ because the top and the bottom are fixed, both of which have just one point having a vertical strand.
By  diagram we see that  the action of $\phi([r_1r_0r_1\ldots]_s)$ in  $W(\ddA_{2m-2})$ (which we identify with $\rm{Sym}_{2m-1}$)
on a number $1\leq a\leq 2m-1$, for $s\leq 2m-2$, is given by
 \begin{equation}\label{odds}\phi([r_1r_0r_1\ldots]_s)(a)=\begin{cases}
a+s,\, a\, {\rm odd},\, a+s\leq 2m-1\\
a-s, \, a\, {\rm even},\, a-s\geq 1\\
4m-1-a-s,\,a\,{\rm odd},\, a+s>2m-1\\
1+s-a, \, a\, {\rm even},\, a-s< 1\\      \end{cases}
\end{equation}
Then  for $t>0$,
$$\phi([r_1r_0r_1\ldots]_{2m-2})(2t)=2m-1-2t,$$
$$ \phi([r_1r_0r_1\ldots]_{2m-2})(2t+1)=2m-1-2t+1. $$
Therefore
$$\phi([r_1r_0r_1\ldots]_{2m-2}^{-1})\alpha_{2t}=\alpha_{2m-1-2t},$$
$$\phi([r_1r_0r_1\ldots]_{2m-2}^{-1})\{\alpha_{2t}\}_{t=1}^{m-1}=\{\alpha_{2t-1}\}_{t=1}^{m-1},$$
which implies that $\phi([r_1r_0r_1\ldots]_{2m-2}^{-1})\phi(e_0)\phi([r_1r_0r_1\ldots]_{2m-2})=\phi(e_1)$ and
(\ref{0.2.6}) holds for the generator images under $\phi$.

Now consider $\phi$ when $n=2m>5$ even.  The fact that (\ref{0.1.3})--(\ref{0.1.6}), and (\ref{0.1.10})--(\ref{0.1.20}) still hold for the generator images under
$\phi$ can be proved easily by Brauer diagrams. As above we see that
\begin{equation}\label{even10s}\phi([r_1r_0r_1\ldots]_s)(a)=\begin{cases}
a+s,\, a\, {\rm odd},\, a+s\leq 2m-1\\
a-s, \, a\, {\rm even},\, a-s\geq 1\\
4m+1-a-s,\,a\,{\rm odd},\, a+s>2m-1\\
1+s-a, \, a\, {\rm even},\, a-s< 1,\\      \end{cases}\end{equation}
\begin{equation}\label{even01s}\phi([r_0r_1r_0\ldots]_s)(a)=\begin{cases}
a+s,\, a\, {\rm even},\, a+s\leq 2m-1\\
a-s, \, a\, {\rm odd},\, a-s\geq 1\\
4m+1-a-s,\,a\,{\rm even},\, a+s>2m-1\\
1+s-a, \, a\, {\rm odd},\, a-s< 1.\\      \end{cases}\end{equation}
By diagram inspection, we see that
$$\phi([r_1r_0r_1\ldots]_{2m-1}^{-1})(\alpha_{2t})=\alpha_{2m-2t}, $$
$$\phi([r_1r_0r_1\ldots]_{2m-1}^{-1})(\{\alpha_{2t}\}_{t=1}^{m-1})=\{\alpha_{2t}\}_{t=1}^{m-1},$$
hence $\phi([r_1r_0r_1\ldots]_{2m-1})\phi(e_0)\phi([r_1r_0r_1\ldots]_{2m-1})=\phi(e_0)$; therefore
relation (\ref{0.1.7}) holds for the images of  the generators  under $\phi$.
On the other hand,
$$\phi([r_0r_1r_0\ldots]_{2m-1}^{-1})(\alpha_{2t-1})=\alpha_{2m-2t+1}, $$
$$\phi([r_0r_1r_0\ldots]_{2m-1}^{-1})(\{\alpha_{2t-1}\}_{t=1}^{m})=\{\alpha_{2t}\}_{t=1}^{m-1}.$$
Just observing  Brauer diagrams, (\ref{0.1.8}) and (\ref{0.1.9}) hold  for the images of  the generators  under $\phi$.

As for (\ref{0.1.13})--(\ref{0.1.15}),
 consider  the top and the bottom of the diagram of the left under $\phi$; both  the top and the bottom have horizontal the same strands among
 those points $\{(i,0)\}_{i=2}^{2m-1}\cup \{(i,1)\}_{i=2}^{2m-1}$ as $\phi(e_0)$.
  Except those $2m-2$ strands in the top
 and in the bottom fixed for $\phi(e_0)$, there are still  two strands of the left side under $\phi$ unknown.  Those two strands are between  the remaining four points, $(0,1)$, $(0, 2m)$, $(1,1)$ and $(1,2m)$.
 If we find another end of the strand from $(1,1)$, the other strand is fixed as a consequence.
 The strands starting from $(1,1)$ in the images of under $\phi$ the right hand sides of   (\ref{0.1.13})--(\ref{0.1.15})
  are ended at $(0,2m)$, $(0,1)$ and $(1,2m)$, respectively.
  By observation,  we can transform this equality problem to the elementary  problem  solved at the beginning of this section in the following way.
Consider the paths of a particle starting from $(1,1)$ in the diagram of the left hand sides of the images under $\phi$ of (\ref{0.1.13})--(\ref{0.1.15}) with the $m-1$ horizontal strands at the top and the
$m-1$ horizontal strands at the bottom removed and transform the horizontal strands as in Figure \ref{trhost}.
  We give an example for this in Figure \ref{I28} for $\phi(e_0)\phi(r_1)\phi(r_0)\phi(r_1)\phi(e_0)$ in $\Br(\ddA_7)$.
   By observation,
   Lemma \ref{km} can be applied here, and gives that  the three equalities hold under
 $\phi$ acting on both sides.
  \begin{figure}[!htb]
\includegraphics[width=.8\textwidth,height=.2\textheight]{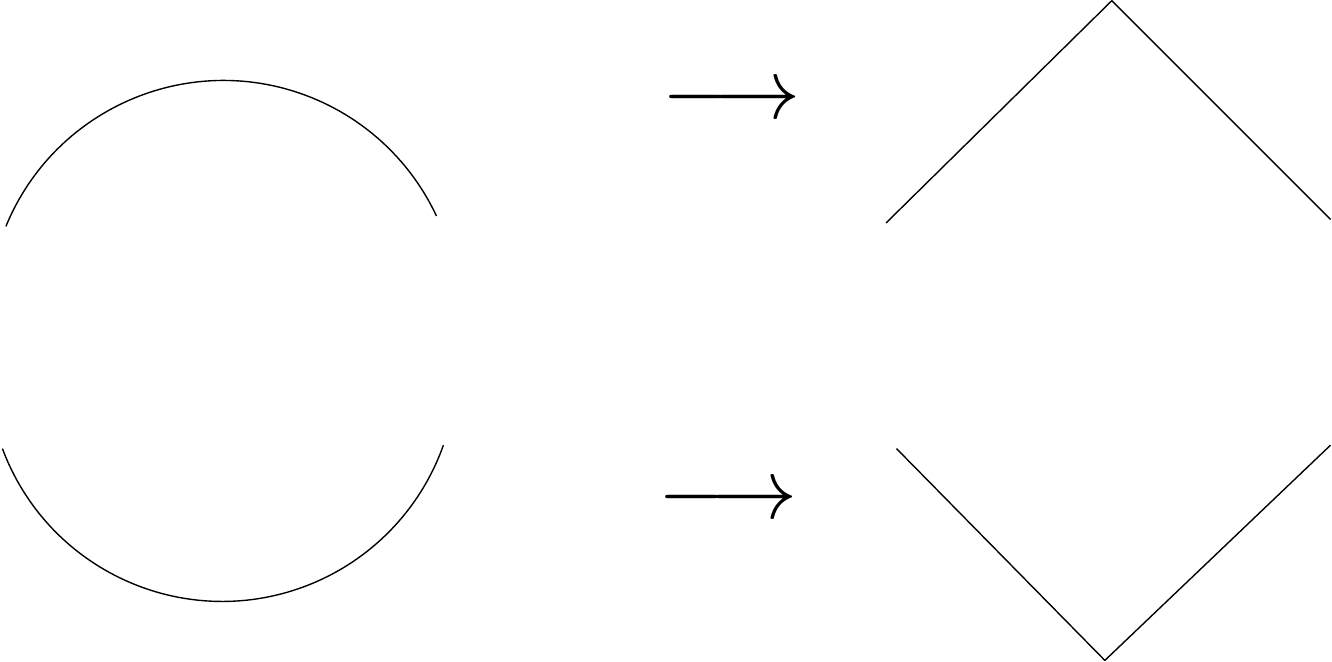}
\caption{transformation of horizontal strands}\label{trhost}
\end{figure}
 \begin{figure}[!htb]
\includegraphics[width=.8\textwidth,height=.3\textheight]{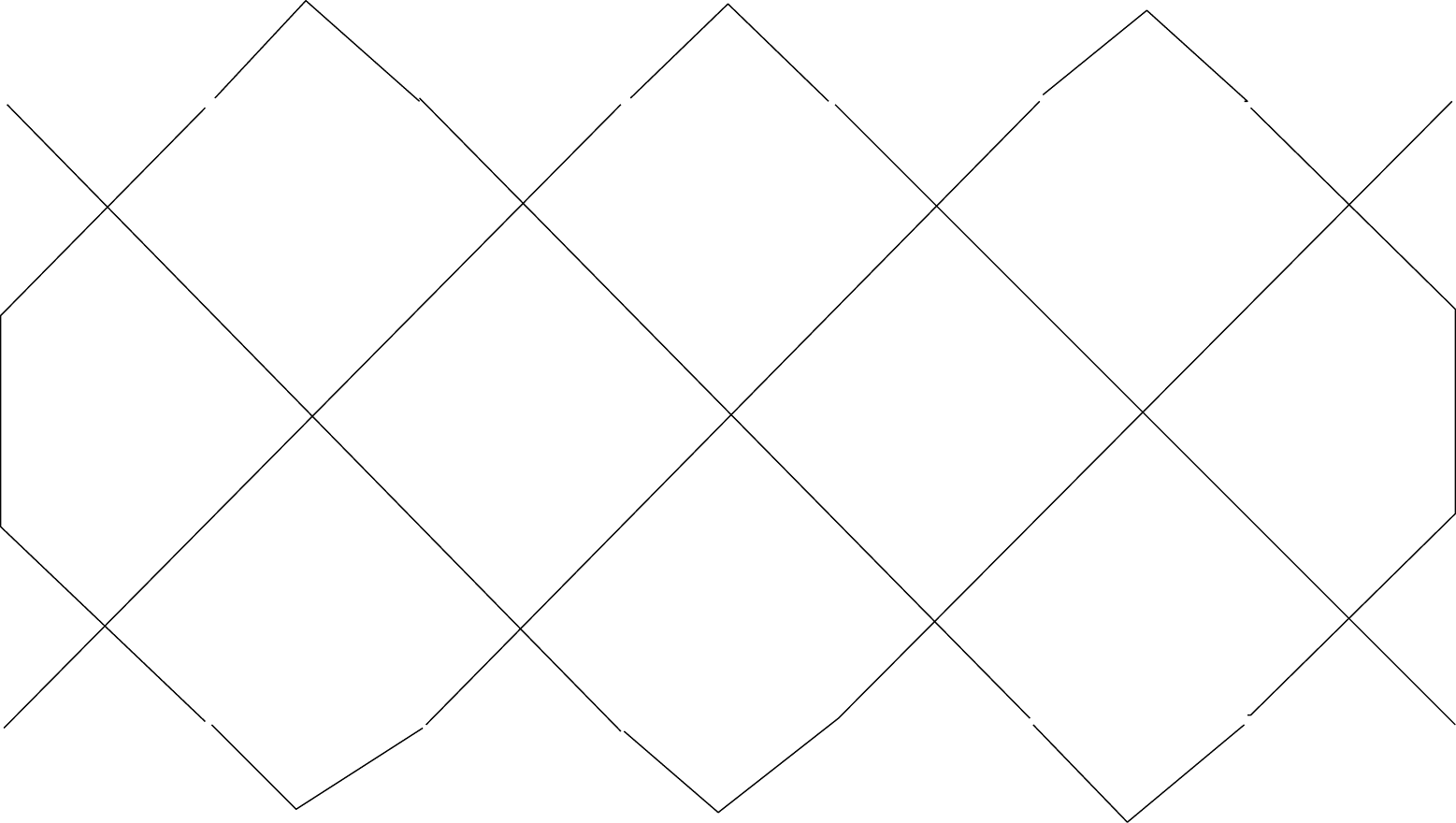}
\caption{$\phi(e_0)\phi(r_1)\phi(r_0)\phi(r_1)\phi(e_0)$ for $\ddI_2^8$}\label{I28}
\end{figure}
\end{proof}

\section{Normal forms for $\BrM(\ddI_2^n)$}
\begin{lemma} The submonoid generated by $r_0$ and $r_1$ in $\BrM(\ddI_2^n)$ is
isomorphic to $W(\ddI_2^n)$.
\end{lemma}
\begin{proof} The lemma follows from  the natural homomorphisms chain below.
$$\Z[\delta^{\pm 1}](W(\ddI_2^{n}))\rightarrow \Br(\ddI_2^n)\rightarrow\Br(\ddI_2^n)/(e_0, e_1)\rightarrow\Z[\delta^{\pm 1}](W(\ddI_2^{n})).$$
The composition is the identity and so the lemma follows.
\end{proof}
From now on, we do not distinguish the $W(\ddI_2^{n})$ in $\BrM(\ddI_2^n)$
and its image under $\psi$.\\
Analogue to Remark \ref{rem:opp}, we can also define an anti-involution $\Br(\ddI_2^n)$ denoted by $x\mapsto x^{\rm op}$.
\begin{prop} \label{prop:opp}
The natural anti-involution above induces an automorphism of the $\Z[\delta^{\pm 1}]$-algebra $\Br(\ddI_2^n)$.
\end{prop}
\begin{proof} It suffices to check the defining relations given in Definition \ref{0.1} and Definition \ref{0.2}
still hold under the anti-involution.  An easy inspection shows that all
relations involved in the definition are invariant under $\op$, except for
(\ref{0.1.8}), (\ref{0.1.9}), (\ref{0.1.10}), (\ref{0.1.11}), and
(\ref{0.1.13}).  The relation obtained by applying the anti-involution to (\ref{0.1.13})
holds due to (\ref{0.1.7}).  The equalities (\ref{0.1.8}) and
(\ref{0.1.10}) are the op-duals of (\ref{0.1.9}) and (\ref{0.1.11}), respectively. Hence our claim holds.
\end{proof}
\begin{prop}\label{odd}Suppose that $D_{0}^{2m-1}$ is the left coset of the
subgroup generated by $r_0$ in $W(\ddI_2^{2m-1})$. Up to some powers
of $\delta$, each element in $\BrM(\ddI_2^{2m-1})$ can be written as an element  in $W(\ddI_2^{2m-1})$ or
$ue_0v$,  where $u\in D_0^{2m-1} $ and $v\in (D_0^{2m-1})^{\rm op}$.
\end{prop}
\begin{proof} By (\ref{0.2.6}), it follows that $e_1$ is conjugate to $e_0$ under $W(\ddI_2^{2m-1})$;
hence we only need to prove that  $D_0^{2m-1}e_0 (D_0^{2m-1})^{\rm op}$ is closed under multiplication  by
the generators $e_0$, $r_0$ and $r_1$ up to some power of $\delta$. By Proposition \ref{prop:opp}, and invariance of the set
  $D_0^{2m-1}e_0 (D_0^{2m-1})^{\rm op}$ under  the natural involution, it suffices to prove it is closed under
  left multiplication.
  For $r_0 $ and $r_1$,  we just apply
(\ref{0.2.4}).
For $e_0$,
  it follows from (\ref{0.2.7}).
\end{proof}

Let $\Psi$ be the root system of $\ddI_2^n$, and $\Psi^+$ be the positive roots with respect to
  $\beta_0$, $\beta_1$  which are  roots corresponding to $r_0$, $r_1$,  respectively. We consider the
 natural action of $W(\ddI_2^n)$ on $\Psi^+$ by negating the negative roots once these appear in this action.

\begin{lemma}\label{N_i}
 Let $N_0$, $N_1$ be  stabilizers of $\beta_0$ and $\beta_1$ in $W(\ddI_2^{2m})$, respectively.
Then for any element $a\in N_i$, we have that $ae_ia^{-1}=e_i$, for $i=0$, $1$.
\end{lemma}
\begin{proof} When $n$ is odd, $N_i=\left< r_i \right>$ for $i=0$, $1$.
Hence the lemma holds because
$$r_i e_ir_i\overset{(\ref{0.2.4})}{=}e_i.$$

When $n=2m$ even, we have that  $N_0=\left<r_0, [r_1r_0\cdots]_{2m-1}\right>$
and $N_1=\left<r_1, [r_0r_1\cdots]_{2m-1}\right>$, hence the lemma holds thanks to the following equalities.
\begin{eqnarray*} r_ie_ir_i&\overset{(\ref{0.1.4})}{=}&e_i,
\\ \left[r_1r_0\cdots]_{2m-1}e_0[r_1r_0\cdots\right]_{2m-1}&\overset{(\ref{0.1.7})}{=}&e_0,
\\ \left[r_0r_1\cdots]_{2m-1}e_1[r_0r_1\cdots\right]_{2m-1}&\overset{(\ref{0.1.8})+(\ref{0.1.9})}{=}&e_1.
\end{eqnarray*}

\end{proof}

Consider a positive root $\beta$ and a node $i$ of type $\ddI_2^n$. If there
exists $w\in W$ such that $w\beta_i=\beta$, then we can define the element
$e_{\beta}$ in $\BrM(\ddI_2^n)$ by
$$e_{\beta}=we_iw^{-1}.$$
The above lemma implies that $e_\beta$ is well defined.

\begin{lemma}\label{anyr}
Let $D_i^{2m}$ be a left coset representatives for $N_i$ in $W(\ddI_2^{2m})$ for $i=0$, $1$,
and $K_0=\langle [r_1r_0\cdots]_{2m-1}\rangle\subset N_0$, $K_1=\langle1\rangle\subset N_1$.
Then for any $r\in W(\ddG_2)$, there exist $a\in D_i$ and  $b\in K_i$,
such that
 $$re_i=ae_ib.$$
\end{lemma}
\begin{proof}
It is a direct result for $(\ref{0.1.4})$, $(\ref{0.1.7})$ and $(\ref{0.1.9})$.
\end{proof}

\begin{prop}\label{even}Up to some power of $\delta$,  each element in $\BrM(\ddI_2^{2m})$ can be written as
\begin{enumerate}[(i)]
\item  $a\in W(\ddI_2^{2m})$,
\item  $u e_iv w$ , $u\in D_i^{2m}$, $v\in K_i$, $w\in (D_i^{2m})^{\rm op}$ for $i=0,\,1,$
\item  $u' e_0 e_1w'$,  $u'\in D_0^{2m}$,  $w'\in (D_1^{2m})^{\rm op}$,
\item  $u'' e_1 e_0w''$, $u''\in D_1^{2m}$,  $w''\in (D_0^{2m})^{\rm op}$,
\item  $u''' e_0 e_1e_0w'''$,$u'''\in D_0^{2m}$,  $w'''\in (D_0^{2m})^{\rm op}$.
\end{enumerate}

\end{prop}
\begin{proof}

Let us first prove  the claim that the monomial $e_0re_1$  can be written as
$e_0e_1$ for any $r\in W(\ddI_2^{2m})$ up to some  power of $\delta$. In view of  (\ref{0.1.4}), we only need  consider the elements that can
be written as $[r_1r_0\cdots]_{2k}$ with $2k\le 2m$. Also thanks to (\ref{0.1.10}), we can restrict ourselves to $m<2k\leq 2m$,   follows
from the below.
\begin{eqnarray*}e_0[r_1r_0\cdots]_{2k}e_1
&\overset{(\ref{0.1.3})}{=}& e_0[r_0r_1\cdots]_{2m-2k-1}[r_0r_1\cdots]_{2m-1}e_1\\
&\overset{(\ref{0.1.9})}{=}& e_0r_0[r_1r_0\cdots]_{2m-2k-2}e_1\\
&\overset{(\ref{0.1.4})}{=}&  e_0[r_1r_0\cdots]_{2m-2k-2}e_1\\
&\overset{(\ref{0.1.10})}{=}& \delta^{\theta_{m-k-1}} e_0e_1.
\end{eqnarray*}

It follows that for any  $r\in W(\ddI_2^{2m})$, the monomial $e_ire_j$  can be
written as  one of $e_0$, $[r_1r_0\cdots]_{2m-1}e_0$ $e_1$, $e_0e_1$, $e_1e_0$, and $e_0e_1e_0$ up to some  power of $\delta$.

To prove the lemma, it remains to prove that those five kinds of normal forms are closed under multiplication by  generators $e_i$. Thanks to Proposition \ref{prop:opp}, we only need consider multiplication from the left. By the conclusion from the above paragraph, the lemma holds.
\end{proof}
\section{The rank of $\rm{Im}\phi$}

To prove Theorem \ref{main}, it suffices to prove that those rewritten forms in Proposition
\ref{odd} and \ref{even} are different
diagrams in $\Br(\ddA_{n-1})$.
The problem can be reduced to the counting of orbit sizes.
\\ \rm{(I)}\, If $n=2m-1$ odd, then
\begin{eqnarray*}
 \#\phi(W(\ddI_2^{2m-1}))(\{\alpha_{2t} \}_{t=1}^{m-1})=2m-1,
 \end{eqnarray*}
\\ \rm{(II)} \,If $n=2m$  even, then
\begin{eqnarray*}
\#\phi(W(\ddI_2^{2m}))(\{\alpha_{2t} \}_{t=1}^{m-1})&=&m,\\
\#\phi(W(\ddI_2^{2m}))(\{\alpha_{2t-1} \}_{t=1}^{m})&=&m,\\
\#\phi(W(\ddI_2^{2m}))(\phi(e_0)\{\alpha_{2t-1} \}_{t=1}^{m})&=&m,
\end{eqnarray*}
and the last two orbits are different.
\\First we consider $n=2m-1\geq 5$.
\\ By  (\ref{odds}), when  $0<s<2m-1$, we see  that $s+1$ is not
occupied  in the horizontal strands of $\phi([r_1r_0r_1\cdots]_s^{-1})(\{\alpha_{2t} \}_{t=1}^{m-1})$. Then
 $\#\phi(W(\ddI_2^{2m-1}))(\{\alpha_{2t} \}_{t=1}^{m-1})$ is at least  $2m-1$.
But the subgroup $\left<r_0\right>$ stabilizes $\{\alpha_{2t} \}_{t=1}^{m-1}$; therefore by Lagrange's Theorem,
{\rm(I)} holds.

When $n=2m>5$, we define
 \begin{eqnarray*}\alpha&=&\Sigma_{i=1}^{2m-1}\alpha_i,\\
 Y_0&=&\{\alpha_{2t} \}_{t=1}^{m-1},\\
 Y_1&=&\{\alpha_{2t-1} \}_{t=1}^{m},\\
 Y_2&=&\phi(e_0)Y_1=Y_0\cup \{\alpha\}.
 \end{eqnarray*}

Now we consider the case when $m=2m'+1$ is odd.
Here we denote by $h(\gamma)$  the height of  $\gamma\in \Phi^+$
 which means the sum of the coefficients of simple roots for $\gamma$ written as the linear combination of
 simple roots.
We find that when $0\leq s\leq m-2$,
$$\max \{h(\gamma)\mid \gamma\in \phi([r_1r_0\cdots]_s^{-1})(Y_0)   \}
= h(\phi([r_1r_0\cdots]_s^{-1})(\alpha_{2m'}))=2s+1,$$
  $$\max \{h(\gamma)\mid \gamma\in \phi([r_1r_0\cdots]_{2m'}^{-1})(Y_0)   \}
= h(\phi([r_1r_0\cdots]_{2m'}^{-1})(\alpha_{2m'}))=4m'.$$
Then it follows that $\#W(\ddI_2^{4m'+2})(Y_0)$ is at least $m$. At the same time
$\left<r_0, [r_1r_0\cdots]_{2m-1}\right>$ stabilizes $Y_0$; therefore it follows from  Lagrange's Theorem
 that $\#W(\ddI_2^{4m'+2})(Y_0)$ is exactly $m$.
 \\Similarly we see that when $0\leq s\leq m-1$
 $$\max \{h(\gamma)\mid \gamma\in \phi([r_0r_1\cdots]_s^{-1})(Y_1)   \}
= h(\phi([r_0r_1\cdots]_s^{-1})(\alpha_{m}))=2s+1.$$
Thus $W(\ddI_2^{4m'+2})(Y_1)$ has at least $m$ elements. At the same time
$\left<r_1, [r_0r_1\cdots]_{2m-1}\right>$ stabilizes $Y_1$, and so by Lagrange's Theorem,
 $\#W(\ddI_2^{4m'+2})(Y_1)=m$.
 \\ When we consider   $W(\ddI_2^{4m'+2})(Y_2)$, we need some result
  from \cite[Section 5]{CLY2010}, we see that $Y_2$ consists of $m'$ symmetric pairs and $1$ symmetric roots, and
  this numerical information is not changed under $W(\ddI_2^{4m'+2})\subset W(\ddC_{2m'+1})$ (Weyl group of type $\ddC_{2m'+1}$ in \cite[Section 5]{CLY2010}.
 The orbit $W(\ddI_2^{4m'+2})(\{\alpha\})$ has at least $m$ elements,
 hence using the same argument as the above, we see that   $\#W(\ddI_2^{4m'+2})(Y_2)$ is also exactly $m$.
 \\ To prove that the orbits of $Y_1$ and $Y_2$ have no intersection, it suffices to verify that $Y_2$ is not in the orbit
 of $Y_1$.
 By the above, we see that $\alpha$ only occurs  in $Y=\phi([r_0r_1\cdots]_{2m'})(Y_1)$  in  the orbit of $Y_1$ under $W(\ddI_2^{2m})$.
 But  $h(\phi([r_0r_1\cdots]_{2m'})(\alpha_{m-2}))=4m'-2>1$, which contradicts the heights of elements in $Y\setminus\{\alpha\}$.
With {\rm (II)} verified, we have proved the Theorem \ref{main} for $n\equiv 2$ $\mod\, 4$, and $n\geq 5$.
\\ At last, consider the case when  $n=2m\geq 5$, and $m=2m'$.
The  formula  $\#W(\ddI_2^{4m'})(Y_1)=\#W(\ddI_2^{4m'})(Y_0)=m$ can be proved by the same argument as the above.
\\From \cite[Section 5]{CLY2010}, we see that $Y_1$ has $m'$  pairs of  symmetric roots and no symmetric root,
and $Y_2$ has $m'-1$ pairs and $2$ symmetric roots $\alpha$ and $\alpha_{2m'}$. Hence the $W(\ddI_2^{2m})$-orbits of $Y_1$ and
$Y_2$
have no intersection.
\\When $0\leq s<2m'-3$, we have

$$\max\{h(\gamma)\mid \gamma\in \phi([r_1r_0\cdots]_{s+1}^{-1})(Y_2\setminus\{\alpha_{2m'}, \alpha\}) \}$$
$$>\max\{h(\gamma)\mid \gamma\in \phi([r_1r_0\cdots]_{s}^{-1})(Y_2\setminus\{\alpha_{2m'}, \alpha\}) \},$$

so  the orbit of $Y_2$ has at least $2m'-2$ elements.
Therefore the cardinality of the stabilizer in $W(\ddI_2^{m})$ is smaller than $\frac{8m'}{2m'-2}$.
If $m'>3$, $\frac{8m'}{2m'-2}<6$, but  the group $\left<r_0, [r_1r_0\cdots]_{4m'-1}\right>$ stabilizes $Y_2$, hence
  the subgroup will be the full stabilizer. By checking when $m'=2$, $3$,
finally, we see that $\#W(\ddI_2^{4m'})(Y_1)=2m'=m$.
With {\rm (II)} verified,
 we have proved the main theorem for $n\equiv 0$ $\mod\, 4$, and $n\geq 5$.

Now we have the following decomposition of $\Br(\ddI_2^n)$ as a $\Z[\delta^{\pm 1}]$-module.
\begin{eqnarray*}
\Br(\ddI_2^{n})&=&\Br(\ddI_2^n)/(e_0)\oplus(e_0), \quad 2\nmid n, \\
\Br(\ddI_2^{n})&=&\Br(\ddI_2^n)/(e_0)\oplus(e_0)/(e_0e_1e_0)\oplus(e_0e_1e_0), \quad 2\mid n.
\end{eqnarray*}
Therefore the  theorem below  about the cellularity can be obtained by an  argument similar to \cite{BO2011}.
\begin{thm}\label{thm.cell}
If $R$ is a field such that   the group ring $R[W(\ddI_2^n)]$ is a cellular algebra, then
the algebra $\Br(\ddI_{2}^n)\otimes R$ is a cellularly stratified algebra.
\end{thm}
\begin{rem}For the hypothesis of the Theorem \ref{thm.cell}, with the method in \cite{G2007}, we conjecture that a sufficient condition for $R[W(\ddI_2^n)]$ is a cellular algebra is that  the characteristic of $R$ does not divide $n$.
\end{rem}
{\bf Acknowledgment} I thank Professor  A . M.Cohen for his reading and giving precious suggestion.

\begin{center}
Shoumin Liu\\
University of Amsterdam\\
email: s.liu@uva.nl\\
liushoumin2003@gmail.com
\end{center}

\end{document}